\documentclass[12pt,a4paper]{article}
\usepackage[T1]{fontenc}

\usepackage{amsxtra}
\usepackage{amsfonts,amsbsy,dsfont}
\usepackage{latexsym}
\usepackage{amsmath}
\usepackage{amssymb}
\usepackage{amscd}
\usepackage{color}
\usepackage{verbatim}
\usepackage{dsfont}
\usepackage{graphics}
\usepackage{graphicx}
\usepackage{epsfig}
\usepackage{lscape}

\renewcommand{\P}{{\rm P}}
\newcommand{\E}{{\rm E}}
\renewcommand{\sp}{\mathrm{sp}}

\newcommand{\norm}[1]{\| #1 \|}

\newcommand{\vc}[1]{\boldsymbol{#1}}

\newcommand{\diag}{\mathrm{diag}}

\newcommand{\ud}{\, \mathrm{d}}  

\def\N{\mathbb N}

\def\DD{{\cal D}}

\def\PP{{\cal P}}

\newcommand{\vect}[1]{\boldsymbol #1}
\newcommand{\valpha}{\vect \alpha}
\newcommand{\vbeta}{\vect \beta}
\newcommand{\vgamma}{\vect \gamma}

\newcommand{\vv}{\vect v}
\newcommand{\vh}{\vect h}

\newcommand{\vone}{\vect 1}
\newcommand{\vzero}{\vect 0}

\newcommand{\vphi}{\varphi}

\newcommand{\vligne}[1]{\begin{bmatrix} #1 \end{bmatrix}}

\newtheorem{defn}{Definition}[section]
\newtheorem{lem}[defn]{Lemma}

\newtheorem{thm}[defn]{Theorem}
\newtheorem{cor}[defn]{Corollary}
\newtheorem{rem}[defn]{Remark}

\newtheorem{exemple}[defn]{Example}

\newcommand{\qed}{\hfill $\square$}
\newenvironment{proof}{
      \noindent {\bf Proof }}{\qed
      \vspace{0.25\baselineskip}
}
\newcommand{\debproof}{\begin{proof}}
\newcommand{\finproof}{\end{proof}}

\newcommand{\tp}{^{\mbox{\tiny T}}}
\newcommand{\bs}{\boldsymbol}
\newcommand{\bsautre}{\mathcal}

\usepackage{pifont}

\definecolor{darkmagenta}{rgb}{0.5,0,0.5}
\definecolor{darkgreen}{rgb}{0,0.6,0}
\definecolor{darkblue}{rgb}{0,0,0.6}
\definecolor{darkred}{rgb}{0.8,0,0}
\definecolor{mellow}{rgb}{.847, 0.72, 0.525}

\newcommand{\stirlingone}[2]{\genfrac{[}{]}{0pt}{}{#1}{#2}} 

\begin{document}

\title{On the nature of \\
Phase-Type Poisson distributions}
\date{}
\author{Sophie Hautphenne\footnote{The University of Melbourne, Department of Mathematics and Statistics, Victoria 3010, Australia; sophiemh@unimelb.edu.au}
\and
Guy Latouche\footnote{%
Universit\'e Libre de Bruxelles, D\'epartement d'Informatique,
CP~212, Boulevard du Triomphe, 1050 Bruxelles, Belgium; latouche@ulb.ac.be}
\and
Giang T. Nguyen\footnote{Corresponding author: The University of Adelaide, School of Mathematical Sciences, SA 5005, Australia; giang.nguyen@adelaide.edu.au}
}
 \maketitle

\begin{abstract}
Matrix-form Poisson probability distributions were recently introduced 
as one matrix generalization of Panjer distributions.  We show in this paper that under the constraint that their representation is to be nonnegative, they have a physical interpretation  as extensions of PH distributions, and we name this restricted family \emph{Phase-type Poisson}.  We use our physical interpretation to construct an EM algorithm-based estimation procedure.
\end{abstract}

\noindent
AMS (2010) subject classification: 91B30; secondary 65Q30, 62P05.  \\
{\em Keywords:} Panjer's algorithm, generalized Panjer distributions, compound distributions, EM algorithm, minimal variance, PH-Poisson.


\section{Introduction}
\label{s:introduction}

First appeared in Panjer (1981)\nocite{panje81}, Panjer's algorithm is designed to compute efficiently the density of sums of the form $S = \sum_{1 \leq i \leq N}X_i$, where  the $X_i$s are i.i.d. positive random variables and $N$ is random, with a density $\{p_n\}$ that follows the recurrence relation
\begin{equation}
   \label{e:panjer}
 p_n = p_{n-1} (a+b/n)  \qquad \mbox{for $n \geq 1$},
\end{equation}
$p_0$ being such that $\sum_{n \geq 0}p_n = 1$.  If the $X_i$s are nonnegative integer-valued random variables with density $\{f_n\}$, then the density $\{g_n\}$ of $S$ may be recursively computed as 
\begin{equation}
   \label{e:panjerG}
g_0 = p_0, \qquad g_n = \sum_{1 \leq i \leq n} f_ig_{n-i}(a+ib/n) \qquad \mbox{for $n \geq 1$.}
\end{equation}
This is a very efficient procedure, which has excellent numerical stability properties.

The distributions that satisfy (\ref{e:panjer}) belong to a restricted set of families consisting of Poisson, binomial and negative binomial distributions (see Sundt and Jewell~(1981)\nocite{sj81}). Much effort has been spent to extend Panjer's algorithm to other distributions for~$N$. In particular, its extension to Phase-type (PH) distributions is of great interest: since they are dense in the class of distributions on $\N$, this significantly increases the applicability of Panjer's algorithm.

Phase-type distributions have been introduced by Neuts~(1975) and (1981)\nocite{neuts75,neuts81} and they may be
defined algebraically as follows: consider a sub-stochastic matrix $T$ of order $m$
such that $I-T$ is nonsingular, a density vector $\valpha$ of order
$m$, and define a sequence $\{\vv_n\}$ of row vectors with
\begin{equation}
   \label{e:ph}
\vv_1 = \valpha (I-T), \qquad \qquad \vv_n= \vv_{n-1} T \qquad \mbox{for $n \geq 2$}.
\end{equation}
The density  $p_0  = 1 - \valpha\vone$, $p_n =  \vv_{n} \vone$, for $n \geq 1$, where $\vone$ is a column vector of ones, is said to be of phase-type,
 with representation $(\valpha, T)$. 
There is a clear similarity between (\ref{e:panjer}) and (\ref{e:ph}), which suggests that the recursion (\ref{e:panjerG}) might be adapted to provide an efficient and numerically stable algorithm to compute the density of $S$ when $N$ has a PH distribution.  This is done in two recent papers, Wu and Li~(2010) and Siaw et al.~(2011)\nocite{wl10, swpw11}. The former defines the {\em generalized $(a,b,0)$ family} as
\begin{equation}
   \label{e:abz}
p_n = \vgamma P_n \vone \qquad \mbox{for $n \geq 0$,}
\end{equation}
where the matrices
$\{P_n\}$  of order $m$ are recursively defined as follows: %
\begin{equation}
   \label{e:Pn}
P_n = P_{n-1} (A+\frac{1}{n}B) \qquad \mbox{for $n \geq 1$.}
\end{equation}
The parameters are the matrices $A$, $B$, $P_0$ and the vector $\vgamma$, which is assumed to be nonnegative and normalized, so that  $\vgamma \vone = 1$.  

Siaw et al.~(2011)\nocite{swpw11} define the {\em generalized $(a,b,1)$ family}, the difference being that the  recursion (\ref{e:Pn}) starts at $n=2$, and the parameters are   $A$, $B$, $P_1$ and $p_0$, while the matrix $P_0$ becomes irrelevant.
The PH$(\valpha,T)$ distribution belongs to the generalized $(a,b,1)$ family, with $A=T$, $B=0$, $p_0 = 1 - \valpha \vone$, $\vgamma = (\valpha \vone)^{-1} \valpha$ and  $P_1=(\valpha \vone) (I-T)$. 

The core of the algorithm in Wu and Li~(2010) and Siaw et al.~(2011)\nocite{swpw11,wl10} is the vector recursion
\begin{equation}
   \label{e:wlg}
\vh_n = \sum_{1 \leq i \leq n} f_i \vh_{n-i} (A+ \frac{1}{i}B)
\end{equation}
to replace (\ref{e:panjerG}), with $g_n = \vh_n \vone$.  Ren~(2010)\nocite{ren10} gives an improved algorithm in case $N$ and the $X_i$s themselves are of phase-type.  Finally, we note that PH distributions have rational generating functions, and this is the basis for the adaptation in Eisele~(2006)\nocite{eisel06} of Panjer's algorithm to the case where $N$ is PH. A comparison of the complexity and numerical stability of the algorithms in Eisele~(2006), Ren~(2010), Wu and Li (2010) and Siaw et al.~(2011)\nocite{eisel06,ren10,swpw11,wl10} is outside the scope of the present paper.

We expect the generalized $(a,b,0)$ and $(a,b,1)$ distributions to form a very rich family since they include the PH distributions. However, as we show in the next section, the combination of two matrices in~(\ref{e:Pn}) makes these distributions a bit unwieldy, unless one imposes some simplifying constraint.
In Section~\ref{s:sums}, we show that the series $\sum_{n \geq 0}P_n$ is a key quantity and that, for all practical purpose, it is necessary that the spectral radius of $A$ be strictly less than one in order for the series to converge. Before doing so, we briefly address the issue of the choice of representation, and we adopt one that is slightly different from the representation in~Wu and Li~(2010) and Siaw et al.~(2011)\nocite{wl10, swpw11}.

Next, we assume in Section~\ref{s:commuting} that $A$ and $B$ commute.  As matrices go, this is a very strong constraint, but it considerably simplifies the determination of the generating function and of moments, and it is a property of all the examples in~Wu and Li~(2010) and Siaw et al.~(2011)\nocite{wl10, swpw11}.  In Section~\ref{s:ph-poisson}, we focus our attention on distributions for which $A=0$, $B \geq 0$, and $\vgamma \geq \boldsymbol{0}$.  These distributions are interesting because they form a family totally distinct from PH distributions, yet they are amenable to a Markovian representation.  For that reason, we call them {\em Phase-type Poisson} or PH-Poisson distributions.  This physical interpretation opens the way in Section \ref{s:em-algorithm} to an estimation procedure based on the EM algorithm.

\section{Matrix generating function}
\label{s:sums}

We are concerned with distributions $\{p_n\}$ defined as
\begin{equation}
   \label{e:PnB}
p_n = \vbeta P_n \vone, \qquad \mbox{where} \qquad P_n = \prod_{1 \leq i \leq n} (A+ \frac{1}{i} B) \qquad \mbox{for $n \geq 0$,}
\end{equation}
$A$ and $B$ are matrices of order $m$, and $\vbeta$ is a row vector of size $m$.  We use the convention that for $n=0$, the matrix product in (\ref{e:PnB})  is equal to the identity matrix, so that we may recursively define the $P_n$s as
\begin{equation}
   \label{e:PnC}
P_0=I, \qquad P_n = P_{n-1} (A+ \frac{1}{n} B) \qquad \mbox{for $n \geq 1$.}
\end{equation}
We shall write that $\{ p_n \}$ has the representation $\DD(\vbeta, A, B)$ of order $m$.

This definition calls for a few comments.  First, we assume that the recursion (\ref{e:PnB}) starts with $n=0$.  In other words, we are not concerned in this paper with the possibility that the sequence $\{ p_n \}$ does not conform to the general pattern for small values of $n$.  Instead, we focus our attention,  to a large extent, on the matrices $P_n$.  

Second, our definition is slightly different from that of generalized $(a,b,0)$ distributions in Wu and Li~(2010)\nocite{wl10}, where it is assumed that $\vgamma$ is a stochastic vector ($\vgamma \geq \vzero$, $\vgamma \vone = 1$) and that $P_0$ is a matrix chosen according to the circumstances. The two representations are equivalent as it suffices to define $\vbeta = \vgamma P_0$.  Our reason to prefer (\ref{e:PnB}) is that  we do not find any advantage in requiring that $\vbeta$ should be stochastic when $A$, $B$ and $P_0$ are allowed to be of mixed signs.  Furthermore, our definition involves $m^2$ fewer parameters (the entries of $P_0$) and this savings will prove significant in Section~\ref{s:em-algorithm} when we design an estimation procedure.

Finally, one might use left- instead of right-multiplication and define $P_n =  (A+ \frac{1}{n} B)  P_{n-1}$,  yielding a possibly different family of distributions.  Actually, we shall assume in the next section that $A$ and $B$ commute, so that there would be no difference.

We need to impose some constraints on the representations of these distributions, otherwise very little can be said in general.  To begin with, let us associate a transition graph to the matrices $A$ and $B$: the graph contains $m$ nodes, and there is an oriented arc from $i$ to $j$ if 
 $|A_{ij}| + |B_{ij}| \not= 0$. A node~$j$ is said to be useful if there exists a node $i$ such that there is a path from $i$ to $j$ in the transition graph {\em and} such that 
$\beta_i \not= 0$; $j$ is said to be useless otherwise.
The lemma below shows that one may require without loss of generality that representations are chosen without useless nodes.

\begin{lem}
\label{t:useful}
If the representation $\DD(\vbeta, A, B)$ of order $m$ is such that there exists at least one useless node, then there exists another, equivalent, representation of order $m'$ strictly less than $m$.
\end{lem}
\begin{proof}
Assume that $j$ is a useless node; define $S_1$ to be the subset of nodes containing $j$ and all the nodes $i$ for which there exists a path from $i$ to $j$.  The matrices $A$ and $B$ may be written, possibly after a permutation of rows and columns, as
\[
A = \begin{bmatrix}
A_{1,1} & A_{1,2} \\ 0 & A_{2,2}
\end{bmatrix}
\qquad \mbox{and} \qquad
B = \begin{bmatrix}
B_{1,1} & B_{1,2} \\ 0 & B_{2,2}
\end{bmatrix},
\]
where $A_{1,1}$ and $B_{1,1}$ are indexed by the nodes in $S_1$ and $A_{2,2}$ and $B_{2,2}$ are indexed by the remaining nodes; similarly, we have
\[
P_n = \begin{bmatrix}
(P_n)_{1,1} & (P_n)_{1,2} \\ 0 & (P_n)_{2,2}
\end{bmatrix}
\]
with $(P_n)_{1,1} = \prod_{1 \leq i \leq n} (A_{1,1}+ \frac{1}{i} B_{1,1})$
and 
 $(P_n)_{2,2} = \prod_{1 \leq i \leq n} (A_{2,2}+ \frac{1}{i} B_{2,2})$.

We partition $\vbeta$ in a similar manner and write $\vbeta = \vligne{\vbeta_1 & \vbeta_2}$.  Since $j$ is useless, $\vbeta_1=\vzero$.
It is clear that $\vbeta P_n \vone = \vbeta_2 (P_n)_{2,2} \vone$, so that $\DD(\vbeta_2, A_{2,2}, B_{2,2})$ is an equivalent representation, of order  strictly smaller than $m$.
\end{proof}

The generating function $p(z) = \sum_{n \geq 0} z^n p_n$ may be written as $p(z)= \vbeta P(z;A,B)\vone$, where 
\begin{equation}
   \label{e:Pz}
P(z;A,B) = \sum_{n \geq 0} z^n P_n = \sum_{n \geq 0} z^n \prod_{1 \leq i \leq n} (A + \frac{1}{i} B),
\end{equation}
provided that the series in (\ref{e:Pz}) converges.  We focus our attention on the matrix generating function $P(z;A,B)$ and  we discuss its convergence properties as $z \rightarrow 1$.  A simple condition for $P(1;A,B)$ to be finite is given in the next lemma.

\begin{lem}
   \label{t:spA}
If $\sp(A) < 1$, where $\sp(\cdot)$ denotes the spectral radius, then the series $P(z;A,B)$ converges for $|z|\leq 1$.

If  $A \geq 0$ and $B \geq 0$, then the inequality $\sp(A) < 1$ is both necessary and sufficient.
\end{lem}
\begin{proof}
The convergence radius $R$ of the series in (\ref{e:Pz}) is given by $R^{-1} = \lim \sup_n \sqrt[n]{\| P_n\|}$, where $\|\cdot\|$ is any matrix norm. 
To simplify the notations, we define $C_i = A+ (1/i)B$.  For any consistent norm, $\norm{P_n} \leq \norm{C_1} \norm{C_2} \cdots  \norm{C_n}$.  Furthermore,
$\norm{C_i} \leq \norm{A} + (1/i) \norm{B}$ and for any $\varepsilon >0$, there exists a norm such that $\norm A < \sp(A) + \varepsilon$.

This implies that if $\sp(A) <1$, then there exist $\eta < 1$ and $i^*$ such that $\norm{C_i} < \eta$ for all $i \geq i^*$.  In addition,
\begin{align*}
\norm{P_n}^{1/n} & \leq \norm{C_1 C_2 \cdots C_{i^*}}^{1/n}   (\norm{C_{i^*+1}} \cdots  \norm{C_n})^{1/n}  \\
 & \leq  \norm{C_1 C_2 \cdots C_{i^*}}^{1/n}   \eta^{(n-i^*)/n},
\end{align*}
for $n \geq i^*$, and $\norm{C_1 C_2 \cdots C_{i^*}}^{1/n}   \eta^{(n-i^*)/n} \rightarrow \eta$ as $n \rightarrow \infty$.  We conclude, therefore, that $\lim\sup_n \sqrt[n]{\|P_n\|} \leq \eta$ and $R \geq 1/\eta > 1$, which proves the first claim.

If $A$ and $B$ are non-negative, then $P_n \geq A^n$ and $P(1;A,B) \geq \sum_{n \geq 0} A^n$; since the last series diverges if $\sp(A)\geq 1$, this completes the proof of the second claim. 
\end{proof}

Note that $\sp(A) < 1$ cannot be a {necessary} condition in all
generality: to give one example, if there is some $n^*$ such that $P_n =0$
for all $n > n^*$, then the series in (\ref{e:Pz}) reduces to a finite
sum, and the spectral radius of $A$ has no bearing on its convergence;
such is the case if $B = -n^*A$.

We turn our attention to the derivatives 
\begin{equation}
   \label{e:Mn}
M_n(A,B) = \left. \frac{\partial^n}{\partial z^n} P(z;A,B) \right|_{z=1},
\end{equation}
for $n \geq 1$, assuming that they exist.  In that  case, the factorial moments of the distribution are given by $m_n=\vbeta M_n(A,B) \vone$.  From the proof of  Lemma~\ref{t:spA}, if $\sp(A) < 1$, then $P(z;A,B)$ is a matrix of analytic functions in the closed unit disk, and it 
 is a sufficient condition for the derivatives to be finite at $z=1$.

\begin{lem}
   \label{t:Mn}
The matrices $M_n(A,B)$ are given by
\begin{equation}
   \label{e:Mna}
M_n(A,B) = n! P_n P(1;A,nA+B).
\end{equation}
If $\sp(A) < 1$, then we also have
\begin{equation}
   \label{e:Mnb}
M_n(A,B) = n! P(1;A,B) \widetilde P_n,
\end{equation}
where 
\[
\widetilde P_n = \prod_{1 \leq i \leq n} ((A+ \frac{1}{i} B)(I-A)^{-1}).
\]
\end{lem}
\begin{proof}
We write
\begin{align}
   \nonumber
P_n & = \frac{1}{n!} \prod_{1 \leq i \leq n}(iA+B)  = \frac{1}{n!} (A+B)\prod_{1 \leq i \leq n-1}(iA + (A+B)) \\
   \nonumber   
  & = \frac{1}{n} (A+B) \prod_{1 \leq i \leq n-1} (A + \frac{1}{i} (A+B))
\end{align}
so that
\begin{align}
   \nonumber
\frac{\partial}{\partial z} P(z;A,B) & = \sum_{n \geq 1} n z^{n-1}P_n  = (A+B) \sum_{n \geq 1}  z^{n-1}  \prod_{1 \leq i \leq n-1} (A + \frac{1}{i} (A+B)) \\
   \nonumber 
 & = (A+B) P(z;A,A+B)
\end{align}
and, by induction, 
\begin{align*}
\frac{\partial^n}{\partial z^n} P(z;A,B) & = (A+B) (2A+B) \cdots (nA+B) P(z;A,nA+B) \\
 & = n! P_n P(z;A,nA+B)
\end{align*}
for all $n$, from which (\ref{e:Mna}) results.   

On the other hand, Lemma~1 in Wu and Li~(2010)\nocite{wl10} states that 
\[
\frac{\partial}{\partial z} P(z;A,B) = z \frac{\partial}{\partial z} P(z;A,B) A + P(z;A,B) (A+B).
\]
If $\sp(A)<1$, then
\[
\frac{\partial}{\partial z} P(z;A,B) = P(z;A,B) (A+B) (I-zA)^{-1}
\]
for $|z|\leq 1$, 
from which (\ref{e:Mnb}) readily results by induction.
\end{proof}

This lemma points to the importance of being able to determine the matrix $P(1;A,B)$.  
In some special cases, an explicit expression may be derived but
in general, in the absence of any simplifying feature of the pair $(A,B)$, there does not seem to be an alternative to the brute force calculation of the series $\sum_{n \geq 0} \prod_{1 \leq i \leq n-1} (A + \frac{1}{i} B)$.

\section{Commutative matrix product}
\label{s:commuting}

In this section, we assume that $A$ and $B$  commute and thereby obtain a stronger result than in Section~\ref{s:sums}.  This assumption is satisfied for all examples in Wu and Li~(2010)\nocite{wl10}, where either $A =0$ or $B$ is a scalar multiple of $A$. It is also satisfied if $A$ or $B$ is a scalar matrix $cI$ for some scalar $c$, or if $B = 0$. The latter includes PH$(\boldsymbol{\alpha},T)$ distributions if there exists a solution to the system of linear constraints $\boldsymbol{\alpha}(I - T)T^{n - 1}\boldsymbol{1} = \boldsymbol{\beta}A^n\boldsymbol{1}$ for $n \geq 1$; if $T$ is invertible, then an obvious solution is $\boldsymbol{\beta} = \boldsymbol{\alpha}(I - T)T^{-1}$, $A=T$.

Thus, although it is a restrictive assumption from a linear algebraic point of view, it may be reasonable in the context of stochastic modeling.

\begin{thm}
   \label{t:pzcommute}
If $A$ and $B$ commute, then $P(z;A,B)= e^{(A+B)D(z;A)}$, where 
\begin{align*}
D(z;A) & = z \sum_{n \geq 1} \frac{1}{n} (zA)^{n-1} \\
  & = A^{-1} \log (I-zA)^{-1} \qquad \mbox{if $A$ is nonsingular.}
\end{align*}
Furthermore, if $B=-kA$ for some integer $k \geq 1$, then $P(z;A,B)= (I-zA)^{k-1}$ and $P(1;A,B)$ is finite, otherwise, $P(1;A,B)$ converges if and only if $\sp(A)<1$.
\end{thm}
\begin{proof}
First, we observe that
\begin{equation}
   \label{e:stirling}
\prod_{1 \leq i \leq n} (iA+B) = \sum_{0 \leq i \leq n} \stirlingone{n}{i} (A+B)^i A^{n-i}
\end{equation}
where $\stirlingone{n}{i}$ are Stirling's numbers of the first kind.  If $A$ and $B$ are scalars, then~(\ref{e:stirling}) is a straightforward consequence of the definition of Stirling's numbers in   Knuth~(1968)\nocite{knuth68}, Section 1.2.6, equation (40). To prove the extension to commuting matrices, one proceeds by induction, using 
\[
\stirlingone{n}{0} =0,  \qquad \stirlingone{n}{n} = 1, \qquad \text{and} \qquad \stirlingone{n}{i-1} + n \stirlingone{n}{i}  = \stirlingone{n+1}{i}
\]
for $n \geq 1$. Next, we write
\begin{align}
   \label{e:ta}
P(z;A,B) & = \sum_{k \geq 0} z^k \frac{1}{k!} \sum_{0 \leq i \leq k} \stirlingone{k}{i} (A+B)^i A^{k-i} 
\qquad \mbox{by (\ref{e:Pz}, \ref{e:stirling})} \\
   \label{e:tb}
 & = \sum_{i \geq 0} z^i (A+B)^i \sum_{k \geq i} \frac{1}{k!} \stirlingone{k}{i} (z A)^{k-i}
\end{align}
since, as we show later, we may interchange the order of summation. By equations (25) and (26) in Knuth (1968)\nocite{knuth68},  Section 1.2.9, 
\[
i! \sum_{k \geq i} \frac{1}{k!} \stirlingone{k}{i} x^{k-i} = (\sum_{k \geq 1} \frac{1}{k}  x^{k-1})^i = (x^{-1} \log (1-x)^{-1})^i,
\]
and (\ref{e:tb}) becomes
\[
P(z;A,B) = \sum_{i \geq 0} \frac{1}{i!} ((A+B) D(z;A))^i,
\]
which proves the first claim.

If $B=-kA$, then  
\[
P(z;A,B)=e^{(1-k)zA\sum_{n \geq 1}(zA)^{n-1}/n} = e^{(k-1)\log(I-zA)} = (I-zA)^{k-1}.
\]
Thus, it remains for us to justify the transition from (\ref{e:ta}) to (\ref{e:tb}).  To that end, we show that the series is absolutely convergent if and only if $\sp(A) < 1$.  It is well-known that $\|A^k\| = O(1) \sp(A)^k k^r$ asymptotically as $k \rightarrow \infty$, for some integer $r \geq 0$.  Furthermore, $\stirlingone{k}{i}/k! = O(1) (\log k)^{i-1}/ (i-1)!$, by Theorem 1 in Wilf~(1993)\nocite{wilf93}.  Therefore, 
\[
\lim_{k \rightarrow \infty} \sqrt[k]{ \frac{1}{k!} \stirlingone{k}{i} \|A^{k-i}\| } = \sp(A) \lim_{k \rightarrow \infty} \sqrt[k]{k^r (\log k)^{i-1}} = \sp(A)
\]
so that the series $\sum_{k \geq i} \frac{1}{k!} \stirlingone{k}{i} (z A)^{k-i}$ in~(\ref{e:tb}) absolutely converges in $\|z\|\leq 1$ if and only if $\sp(A)<1$, in which case its limit is $\frac{1}{ i!}(D(z;A)/z)^i$. The equation~(\ref{e:tb}) becomes
\[
P(z;A,B) = \sum_{i \geq 0} \frac{1}{i!} ((A+B) D(z;A))^i
\]
which converges without further constraint.
\end{proof}

This theorem confirms the important role of the matrix $A$ with respect to the convergence of various series.  A direct consequence is that if $A$, $B_1$ and $B_2$ are three commuting matrices, then
\begin{align}
   \nonumber
P(z;A,B_1) P(z;A,B_2) & = e^{(A+B_1)D(z;A)}  e^{(A+B_2)D(z;A)} =e^{(2A+B_1+B_2)D(z;A)}  \\
 & = P(z;A,A+B_1+B_2),
\end{align} 
so that, if $A$ and $B$ commute, we may write that
\begin{align*}
P(z; A, kA+B) &= P(z;A,B) (P(z; A,0))^k   = P(z;A,B) (e^{AD(z;A)})^k \\
 & = P(z;A,B) (I-zA)^{-k}
\end{align*}
for $k \geq 0$, $k$  integer, and we may state the following property, using either~(\ref{e:Mna}) or (\ref{e:Mnb}):

\begin{cor}
   \label{t:Mnb}
If $A$ and $B$ commute, then the $n$th factorial moment of the distribution is given by
\begin{equation}
   \label{e:moment}
m_n(\vbeta, A, B) = n! \vbeta P(1; A,B) (I-A)^{-n} P_n \vone.
\end{equation}
\qed
\end{cor}

If one remembers that $\vbeta P(1; A,B)$ is a vector of which the components add-up to one, the similarity with the factorial moments of discrete PH distributions is striking (see equation (2.15) of Latouche and Ramaswami~(1999)\nocite{lr99}).

To conclude this section, we review the examples in~Wu and Li (2010)\nocite{wl10}:
\begin{itemize}
\item If $B= \alpha A$ for $\alpha \geq -1$, then $(A+B) D(z;A) = (1+\alpha) \log(I-zA)^{-1}$ and $P(z;A,\alpha A)=(I-zA)^{-(1+\alpha)}$.
\item If $B=-k A$ for $k \geq 0$, $k$ integer, $P(z;A,-kA)=(I-zA)^{k-1}$ as proved in Theorem~\ref{t:pzcommute}.
\item If $A=0$, then $D(z;0)=z$ and $P(z;0,B) = e^{zB}$.
\end{itemize}
We shall further examine this last case in the remainder of the paper.

\section{PH-Poisson distributions}
\label{s:ph-poisson}

\subsection{Definition and comparison to PH distributions} \label{subsec:defcomp}

We restrict our attention to distributions for which $A=0$, with the added constraint that $\vbeta \geq \vzero$, $B \geq 0$.  The assumption that $\vbeta$ and $B$ are non-negative makes it easier to ascertain that $\DD(\vbeta,0,B)$ is the representation of a probability distribution. In addition, as we show in Theorem~\ref{t:physical}, it provides us with a physical interpretation in terms of a Markovian process, which explains why we call these {\em Phase-type Poisson} distributions, or PH-Poisson for short.

\begin{defn}
A random variable $X$ has a PH-Poisson distribution with representation  $\PP(\vbeta,B)$ if 
\begin{equation}
   \label{e:mpoisson}
\P[X=n] = p_n = \frac{1}{n!} \vbeta B^n \vone, \qquad \mbox{for $n \geq 0$,}
\end{equation}
where $B \geq 0$ is a matrix of order $m$ and $\vbeta \geq \vzero$ is a row-vector of size $m$ such that $\vbeta e^B \vone = 1$.
\end{defn}

Note that $\vbeta \vone < 1$, unless $B = 0$.   In the notation of Wu and Li~(2010)\nocite{wl10}, the PH-Poisson distribution with representation $\PP(\vbeta,B)$  belongs to the generalized ($a,b,0$) family, with $A=0$, $B=B$, $P_0=e^{-B}$ and $\vgamma = \vbeta e^B$.
The generating  function $p(z)= \sum_{n \geq 0} z^n p_n$ is given by $p(z) = \vbeta e^{zB} \vone$ and  the factorial moments by
\begin{equation}
   \label{e:mpmoments}
\E[X(X-1)\cdots (X-n+1)] = \vbeta B^n e^{B} \vone.
\end{equation}
It is easy to see that PH and PH-Poisson distributions are essentially two different families of probability distributions.  Indeed, assume that $X$ is PH-Poisson with representation $\PP(\vbeta,B)$ and $Y$ is PH with representation PH($\valpha,T$).  From (\ref{e:mpoisson}), it results that
\begin{equation}
   \label{e:B}
\P[X=n] \approx (\sp(B))^n n^r / n!,
\end{equation}
asymptotically as $n \rightarrow \infty$, where $r$ is the index of $\sp(B)$, and similarly,
\begin{equation}
   \label{e:T}
\P[Y=n] = \valpha T^{n-1} (I-T) \vone \approx (\sp(T))^n n^s,
\end{equation}
where $s$ is the index of $\sp(T)$.  It is obvious that for any given $B$ there  is no $T$ such that the right-hand sides of (\ref{e:B}) and (\ref{e:T}) coincide for all $n$ big enough, unless $\sp(B)=\sp(T)=0$.

If $\sp(B)=0$, then there exists $k \leq m$ such that $B^k =0$, the distribution of $X$ is concentrated on $\{ 0, 1, \ldots, k\}$, and $X$ does have a PH representation by Theorem 2.6.5 in Latouche and Ramaswami~(1999)\nocite{lr99}.

\begin{exemple}  \rm {\bf Tail of the density.}  We see  from (\ref{e:B}) that the density of PH-Poisson distributions  drops sharply to zero.  We compare three different densities on Figure~\ref{f:smalldens}: one is a PH-Poisson distribution, the second a PH distribution and the third a Poisson distribution.  We have connected the points of the densities  for better visual appearance, and we plot on the right-hand side the tail of the densities in semi-logarithmic scale.

\begin{figure}[!tbp]
  \centering
     \includegraphics[scale=0.65]{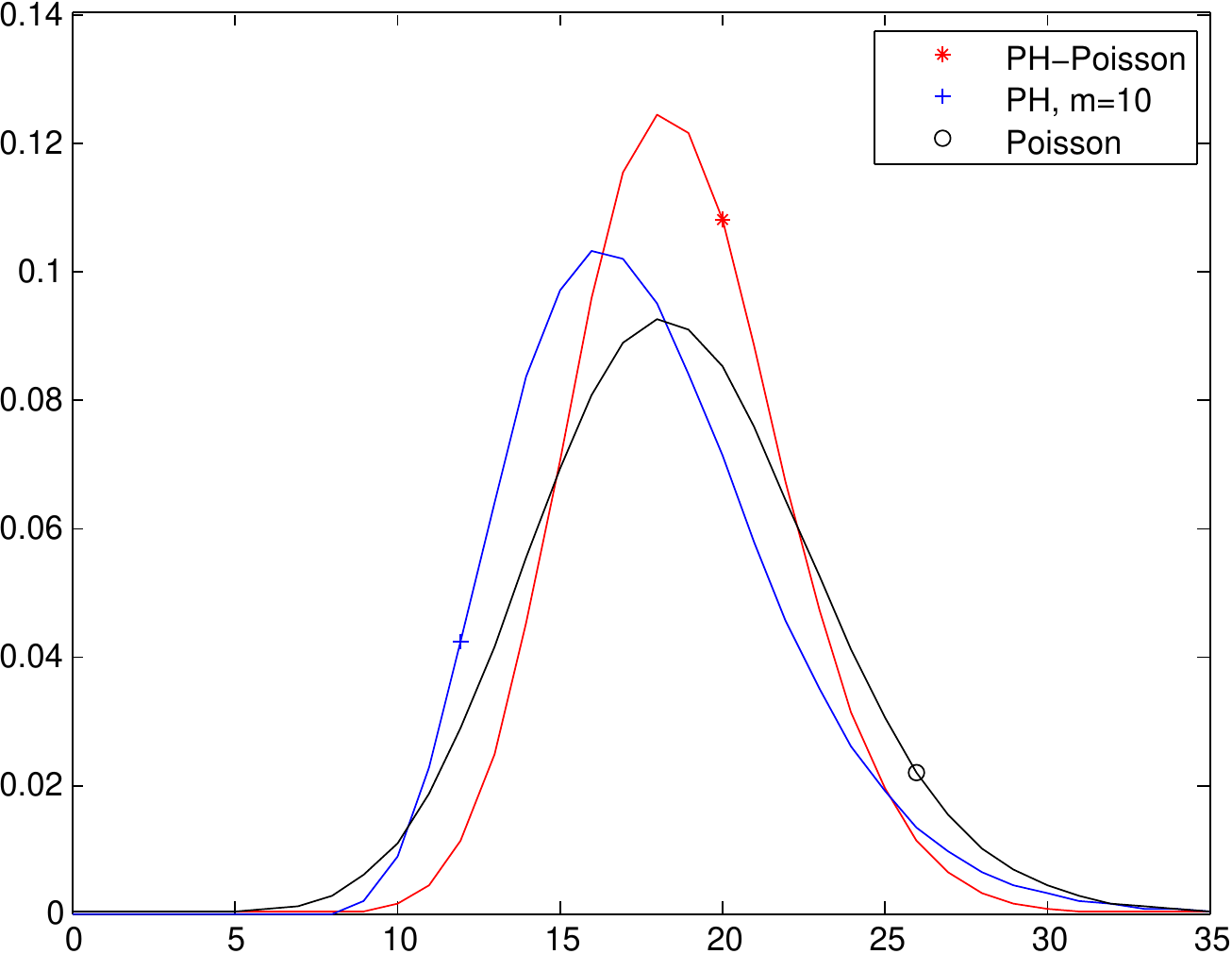}
   \includegraphics[scale=0.65]{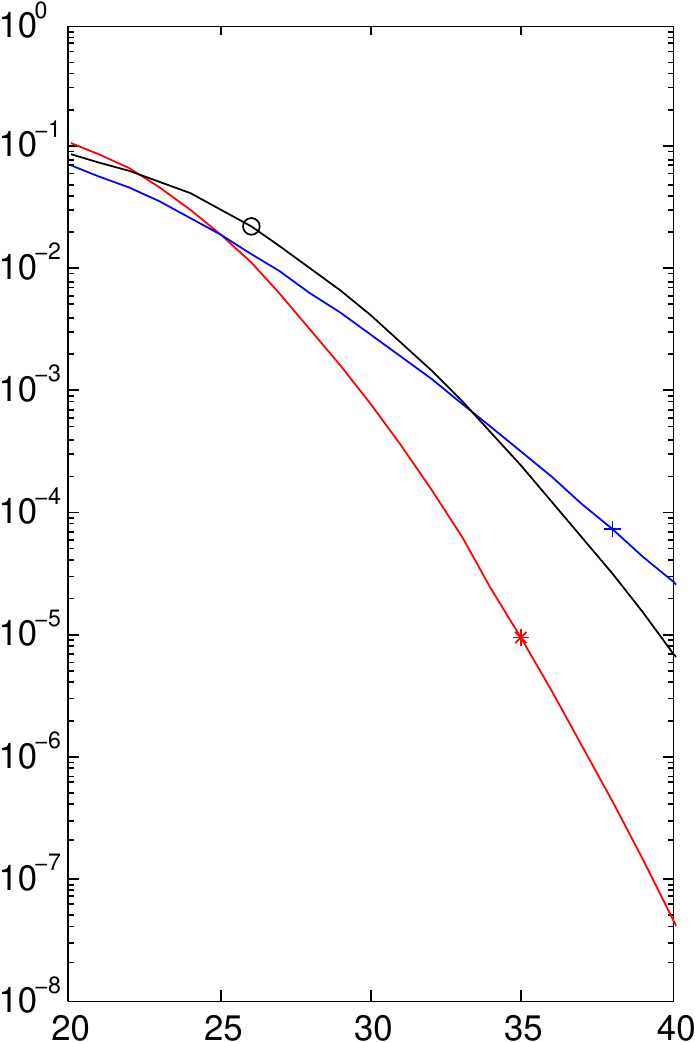}
    \caption{Density function for a PH-Poisson distribution with $m=10$ and mean 18.71 (curve marked with $*$), the Poisson distribution with the same mean (marked with $+$) and the minimal-variance PH distribution with the same order and the same mean  (marked with  $\circ$).}
    \label{f:smalldens}
\end{figure}

The curve marked with a ``$*$'' is the density of the PH-Poisson distribution with $m=10$, $B_{ii}=10$, $1 \leq i \leq m$, and $B_{i,i+1}= 37.5$, $1 \leq i \leq m-1$.  The vector $\vbeta$ is given by
\[
\vbeta = \vligne{1 & 0 & \cdots & 0} (\diag(e^B \vone))^{-1}
\]
(it is easy to verify that $\vbeta e^B \vone = 1$.)  Its mean $\mu$,  variance $\sigma^2$ and  coefficient of variation C.V. equal to $\sigma/\mu$ are given in the first row of Table~\ref{t:smalldens}, as well as the spectral radius S.R. of the matrix $B$.

\begin{table}[h!]
  \centering
  \begin{tabular}{l|cccc}
     & $\mu$ & $\sigma^2$ & C.V. & S.R.\\
\hline
PH-Poisson & 18.71 & 10.35 & 0.17 & 10\\
Phase-type & `` & 16.30 & 0.22 & 0.47 \\
Poisson & `` & 18.71 & 0.23 & 18.71 
  \end{tabular}
\caption{Mean, variance and coefficient of variation of the three distributions of Figure~\ref{f:smalldens}.}    \label{t:smalldens}
\end{table}

The curve marked with a ``$+$'' is the density of the  PH distributions with the same order $m$ and mean $\mu$ and minimal variance (see Telek~(2000)\nocite{telek00} for details).  The curve marked with a ``$\circ$'' is the Poisson density with parameter equal to the mean $\mu$.  The variance,   coefficient of variation, and spectral radius of these two densities are also given in Table~\ref{t:smalldens}.

The plot on the right-hand side of Figure~\ref{f:smalldens} clearly indicates that the PH-Poisson density decays asymptotically the fastest of the three, this is due to the combination of a relatively small spectral radius and of the factor $1/n!$.  We also see from the plot on the left-hand side, and from Table~\ref{t:smalldens}, that it is the most concentrated around the mean.
\end{exemple}

\begin{exemple}  \rm {\bf Small variance.}  We pursue here the comparison between PH-Poisson distributions and minimal variance PH distributions, showing  that PH-Poisson distributions may prove to be a useful alternative to PH distributions when modeling discrete distributions with small variance. 

The curve marked with a ``$*$'' on Figure~\ref{f:smallvar} is the density of the PH-Poisson distribution with $m=10$, $B_{ii}= 2+4 i$, $1 \leq i \leq m$, and $B_{i,i+1}= 0.5$, $1 \leq i \leq m-1$.  The vector $\vbeta$ is given by
\[
\vbeta = \frac{1}{m}\vligne{1 & 1 & \cdots & 1} (\diag(e^B \vone))^{-1}.
\]
Its mean, variance, coefficient of variation and spectral radius are given in Table~\ref{t:smallvar}.

\begin{figure}[!tbp]
  \centering
   \includegraphics[scale=0.65]{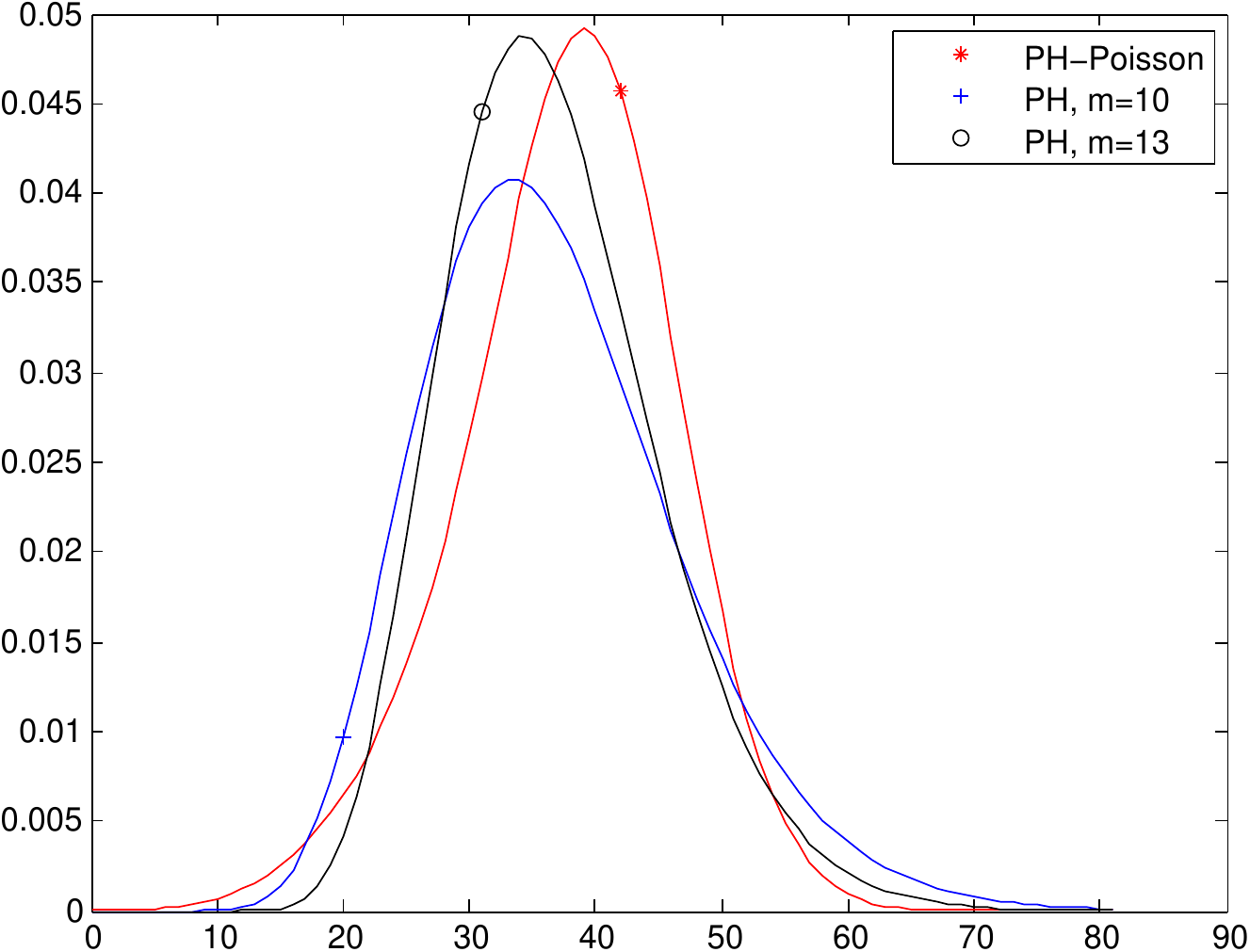}
    \caption{Density function for a PH-Poisson distribution with $m=5$ (curve marked with $*$), a minimal-variance PH distribution with $m=5$ (marked with $+$) and a minimal-variance PH distribution with $m=17$ (marked with  $\circ$). The distributions all have the same mean $\mu=34.00$.}
    \label{f:smallvar}
\end{figure}

The two other curves are the density functions of PH distributions with minimal variance, with the same mean as the PH-Poisson distributions, and with different orders. 
The one marked with ``$+$'' has the same order $m=10$ as the PH-Poisson distribution, the one  marked with  ``$\circ$'' has order $m=13$, the smallest value for which the minimal variance is smaller than that of the PH-Poisson distribution. 

\begin{table}
  \centering
  \begin{tabular}{l|cccc}
     & $\mu$ & $\sigma^2$ & C.V. & S.R.\\
\hline
PH-Poisson & 37.71 & 73.89 & 1.96 & 42\\
Phase-type, $m=10$ & `` & 104.5 & 2.77 & 0.73 \\
Phase-type, $m=13$ & `` & 71.69 & 1.90 & 0.66 
  \end{tabular}
\caption{Mean, variance and coefficient of variation of the three distributions of Figure~\ref{f:smallvar}.}
   \label{t:smallvar}
\end{table}
\end{exemple}

\subsection{A physical interpretation} \label{subsec:physint}

We now give a physical interpretation for PH-Poisson distributions. First, we define the Poisson process \{$\theta_1$, $\theta_2$, \ldots\} of rate $\nu = \max_{i}{(B \vone)_i}$.  Second, we define $P=\nu^{-1} B$; $P$ is a sub-stochastic matrix, possibly stochastic. Next, we consider a discrete PH random variable $K$ with representation $(\valpha, P)$, where $\valpha = c\boldsymbol{\beta}$ for some arbitrary but fixed constant $c \leq (\boldsymbol{\beta}\boldsymbol{1})^{-1}$. In the present description, the Markov chain with transition matrix $P$ makes a transition at each event of the Poisson process and it gets absorbed at time $T=\theta_K$. Finally, we count the number $N(t)$ of transitions between {\em transient} states until the Markov chain enters its absorbing state; that is, $N(t)$ is the number of Poisson events in the interval $(0,t)$ for $t<T$ and $N(t)=K-1$ for $t \geq T$.

\begin{thm}
   \label{t:physical}
If $\valpha = c\boldsymbol{\beta}$ for some arbitrary but fixed constant $c \leq (\boldsymbol{\beta}\boldsymbol{1})^{-1}$
, then $p_n$ defined in (\ref{e:mpoisson}) is the conditional probability
\begin{equation}
   \label{e:physical}
p_n = \P[N(1)=n | T >1].
\end{equation}
\end{thm}
\begin{proof}
Define $M_k(t)$ such that
\[
(M_k(t))_{ij} = \P[N(t)=k, \vphi(t)=j | \vphi(0)=i], \qquad \mbox{for $1 \leq i, j \leq m$}.
\]
One easily verifies that $M_0(t) = e^{-\nu t} I$ and one proves by induction that
\begin{equation}
   \label{e:mmk}
M_k(t) = e^{-\nu t} (\nu P)^k t^k / k! = e^{-\nu t} B^k t^k/ k!,
\end{equation}
for $k \geq 1$.  Equation (\ref{e:mmk})  holds for $k=0$ and we assume that it holds for some $k-1$.  Conditioning on the epoch $u$ of the first Poisson event, we find that
\begin{align*}
M_k(t) & = \int_0^t e^{-\nu u} \nu P M_{k-1}(t-u) \, \ud u \\
 & = \int_0^t e^{-\nu u} B e^{-\nu(t-u)} B^{k-1}  (t-u)^{k-1}/(k-1)!  \, \ud u \\
 & =  e^{-\nu t}  B^{k} /(k-1)! \int_0^t (t-u)^{k-1}  \, \ud u \\
 & =  e^{-\nu t}  B^{k} t^k /k!.
\end{align*} 
Taking $t=1$, we find that
\begin{align*}
\P[N(1)=k, T>1] & = \sum_{1\leq i \leq M}\sum_{1\leq j \leq M} \alpha_i\P[N(t)=k, \vphi(t)=j | \vphi(0)=i] \\
 &= \valpha e^{-\nu}B^k/k! \vone,
\end{align*}
{so that}
$ \P[T>1]  = \valpha e^{-\nu}e^B \vone $,
and
\begin{align*}
 \P[N(1)=k | T >1] & = (\valpha e^{-\nu}e^B \vone)^{-1} \valpha e^{-\nu}B^k/k! \vone\\
 &=  (\valpha e^B \vone)^{-1} \valpha  B^k/k! \vone.
\end{align*}
If $\valpha = c \vbeta$ for any scalar $c \leq (\boldsymbol{\beta}\boldsymbol{1})^{-1}$, then $\P[N(1)=k | T >1] =p_k $ for all~$k$. This concludes the proof.
\end{proof}

\begin{rem} \label{rem:stoch} \rm If $P$ is stochastic, then the random variable $K$ has an unusual PH distribution, as it is either equal to zero or to infinity.  Still, the argument in the proof of Theorem~\ref{t:physical} holds true. Note that if $P$ is stochastic, then $B \vone = \nu \vone$ and the distribution (\ref{e:mpoisson}) is Poisson with parameter~$\nu$. 
\end{rem} 

\begin{exemple} \label{exemple:transition} \rm

This is a PH-Poisson distribution chosen to illustrate the combined effect of the conditional distribution imposed on the number of transitions among the phases.  The representation is $(\vbeta, B)$ with 
\begin{equation}
   \label{e:b0404k}
B= \left[
\begin{array}{ccccc}
   5 &   .05    &     0    &     0     &    0 \\
    .05  &  9  &  .05    &     0     &    0 \\
         0 &   .05&   13  &  .05    &     0 \\
         0  &       0  &  .05 &  17 &   .05 \\
         0    &     0    &     0  &  .05 &  21  
\end{array} \right]
\end{equation}
{and} 
\begin{equation}
   \label{e:beta0404k}
\vbeta = \gamma \vligne{5. & 2.5 & 3. & 2.25 & 6.} \exp\{-\diag(5, 9, 13, 17, 21)\}
\end{equation}
where the scaling factor $\gamma$ is such that $\vbeta e^{B} \vone = 1$.  Its first two moments are $\mu = 13.84 $ and $\sigma^2 = 47.31$, and its density is given on Figure~\ref{f:mp0404k}.
\begin{figure}[!tbp]
  \centering
   \includegraphics[scale=0.67]{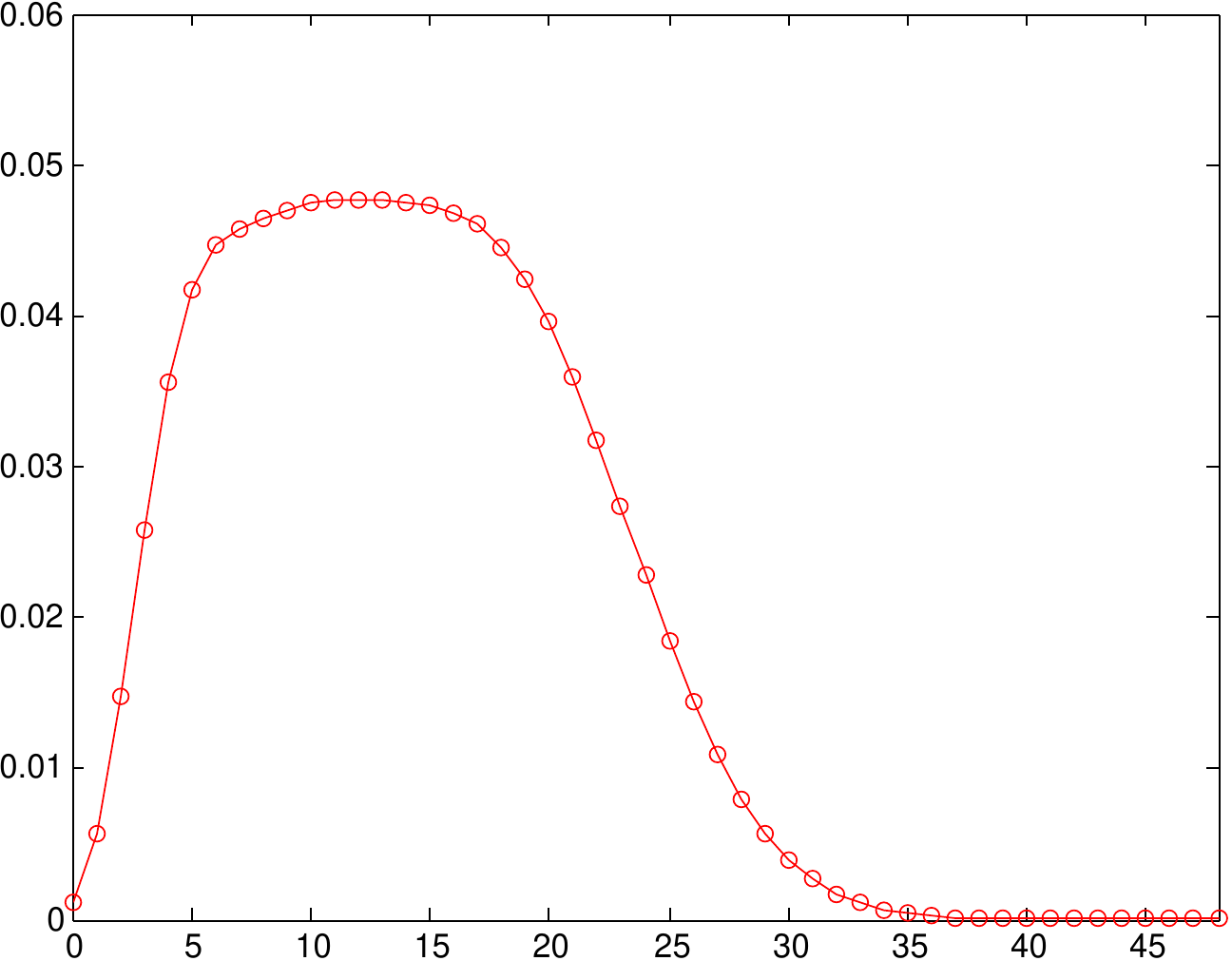}
    \caption{Density function for a PH-Poisson distribution with representation $(\vbeta, B)$ given in (\ref{e:b0404k}, \ref{e:beta0404k}).}
    \label{f:mp0404k}
\end{figure}
The phase-type representation is $(\nu; \valpha, P)$ with $\nu =21.05$,
\begin{equation}
   \label{e:p0404k}
P = \left[
  \begin{array}{ccccc}
    0.2375 &   0.0024   &      0  &       0    &     0  \\
    0.0024 &   0.4276  &  0.0024   &      0    &     0  \\
         0  &  0.0024  &  0.6176   & 0.0024  &       0  \\
         0   &      0  &  0.0024   & 0.8076 &   0.0024  \\
         0    &     0  &       0  & 0.0024  &  0.9976   
  \end{array}
\right]
\end{equation}
and 
\[
\valpha \approx \vligne{0.99 & 0.91 \, 10^{-2} & 0.20 \, 10^{-3} & 0.27 \, 10^{-5} & 0.13 \, 10^{-6}},
\]
which is the vector $\boldsymbol{\beta}$ normalized so that $\valpha \vone =1$.

Denote by  $N^*= \sup\{k:\theta_k<1\}$ the total number of Poisson events in $(0,1)$.  If $T \leq 1$, then $N(1)= K-1 < N^*$,   if $T>1$, then $N(1) = N^*<K$. 
On the average, the Poisson process produces $\nu$ events in the interval $(0,1)$, and the Poisson distribution has a relatively small standard deviation, so one expects $N^*$ to take values close to $\nu \approx 21$.
The matrix $P$ in (\ref{e:p0404k}) is irreducible, albeit with a small probability of migration from one phase to another, so that the initial phase plays a significant role in the distribution of $K$.

If the initial phase is 1, then the absorption probability is about 0.76, and it is likely that $K$ will be small; it is therefore necessary, for the condition $T>1$ to be fulfilled,  that the Poisson process produces few events in $(0,1)$.

On the other hand, if the initial phase is 5, then the PH Markov chain will remain in that phase for a large number of transitions, it is likely that  $K$  will be large, so that $T$ is likely to be much larger than 1, and it is not expected that the condition $[T>1]$ puts much constraint on $N^*$.

\end{exemple}

\section{EM algorithm}
\label{s:em-algorithm}

In this section, we exploit the probabilistic interpretation of PH-Poisson distributions given in Section~\ref{subsec:physint}, and we 
 develop an EM algorithm for fitting PH-Poisson distributions into data samples. 

The EM algorithm is a popular iterative method in statistics for computing maximum-likelihood estimates from data that is considered incomplete. The procedure can be explained briefly as follows. Let $\boldsymbol{\theta} \in \Omega$ be the set of parameters to be estimated. We denote by $\bsautre{X}$ a random complete data sample 
and by $f(\bsautre{X}\,|\,\bs{\theta})$ its conditional density function, given the parameters $\bs{\theta}$. 
The maximum-likelihood estimator $\hat{\bs{\theta}}$ is defined as 
\begin{align*} 
\hat{\bs{\theta}} = \operatorname*{arg\,max}_{\bs{\theta} \in \Omega} \log f (\bsautre{X}\,|\,\bs{\theta}).
\end{align*} 
For one reason or another, instead of observing the complete data sample $\bsautre{X}$, we observe an incomplete data sample $\bsautre{Y}$. {Thus, $\bsautre{X}$ can be replaced by its sufficient statistic $(\bsautre{Y},\bsautre{Z})$, where $\bsautre{Z}$ is the sufficient statistic of the unobserved data.} As $\bsautre{X}$ is unobservable, instead of maximizing $\log f(\bsautre{X}\,|\,\bs{\theta})$ we maximize its conditional expectation given the incomplete data sample {$\bsautre{Y} = \bs{y}$} and the current estimates $\bs{\theta}^{(s)}$, at each $(s + 1)$th iteration for $s \geq 0$.

%

The EM algorithm can thus be decomposed into two steps:
\begin{itemize} 
\item E-step---computing the conditional expectation of $\log f(\bsautre{X}\,|\,\bs{\theta})$ given the incomplete data sample {$\bs{y}$} and the current estimates $\bs{\theta}^{(s)}$ 
\begin{align*} 
Q(\bs{\theta},\bs{\theta^{(s)}}) = \E[\log f(\bsautre{X}\,|\,\bs{\theta})\,|\, {\bs{y}}, \bs{\theta}^{(s)}],
\end{align*}
\item M-step---obtaining the next set $\bs{\theta}^{(s + 1)}$ of estimates by maximizing the expected log-likelihood determined in the E-step
\begin{align*}
\bs{\theta}^{(s + 1)} = \operatorname*{arg\,max}_{\bs{\theta} \in \Omega} Q(\vc\theta,{\vc\theta}^{(s)}).
\end{align*} 
\end{itemize} 
When fitting a PH-Poisson distribution into a data sample, the parameters to be estimated are $\bs{\theta} = \{\nu, \bs{\alpha}, P\}$. Without loss of generality, we assume that $\bs{\alpha}\bs{1} = 1$ in the chosen representation. By Theorem~\ref{t:physical}, an observation $y$ can be thought of as the number of Poisson events in the time interval $[0,1]$, given that the transient Markov chain with the transition matrix $P$ has not been absorbed at time $t = 1$. This observation can be considered incomplete as it tells us neither the initial phase $\varphi(0)$ of the Markov chain nor how it has evolved during $[0,1]$;  a \emph{complete} observation can be represented by 
%
$
 x = (\varphi_0, \varphi_1, \ldots, \varphi_y),
$
%
where $\varphi_i$  is the phase of the Markov chain at the $i$th Poisson event and $\varphi_i \neq 0$ for all $i = 0, \dots, y$.
The conditional density of the complete observation $x$ given~$\bs{\theta}$ is 
\begin{align*} 
f(x\,|\,\bs{\theta}) = (\boldsymbol{\alpha} e^{\nu P} \boldsymbol{1})^{-1}\, \frac{\nu^y}{y!}\, {\alpha_{\varphi_0}\,\prod\limits_{i = 0}^{y - 1} p_{\varphi_i \varphi_{i + 1}}}.
\end{align*} 

Suppose that the complete data sample $\boldsymbol{x}$ contains $n$ observations, each of which is denoted by $x^{[k]}$ and includes an incomplete observation $y^{[k]}$, for $k = 1, \ldots, n$. Then, the conditional density of $\boldsymbol{x}$ given $\vc\theta$ is 
\begin{align*} 
f(\vc x \,|\, \vc\theta) & = (\boldsymbol{\alpha} e^{\nu P} \boldsymbol{1})^{-n}\,  \prod_{k = 1}^{n}\frac{\nu^{y^{[k]}}}{y^{[k]}!}\,\prod_{k = 1}^{n} \alpha_{\varphi_0^{[k]}} \, \prod_{k = 1}^{n} \left(\prod_{i = 0}^{y^{[k]} - 1}p_{\varphi_i^{[k]} \varphi_{i + 1}^{[k]}}\right) \\
& =  (\boldsymbol{\alpha} e^{\nu P} \boldsymbol{1})^{-n}\,  \prod_{k = 1}^{n}\frac{\nu^{y^{[k]}}}{y^{[k]}!}\,\prod_{i = 1}^{m} \alpha_{i}^{S_i} \prod_{i = 1}^{m} \prod_{j = 1}^{m} p_{ij}^{N_{ij}}, 
\end{align*} 
where
\begin{align*} 
S_i=\sum_{k=1}^n \mathds{1}_{\{\varphi_0^{[k]}=i\}} \quad \mbox{for } i = 1, \ldots, m
\end{align*} 
is the number of complete observations in $\boldsymbol{x}$ with initial phase $i$,  and 
\begin{align*} 
N_{ij}=\sum_{k=1}^n\sum_{t\geq1} \mathds{1}_{\{\varphi_{t-1}^{[k]}=i,\varphi_{t}^{[k]}=j \}} \quad \mbox{for } i, j = 1, \ldots, m
\end{align*}
is the total number of jumps in $\bs{x}$ from phase $i$ to phase $j$.
Thus, the log-likelihood function is given by
\begin{align} 
\log f(\vc x \,|\,\vc\theta) & = -n \log( \boldsymbol{\alpha} e^{\nu P} \boldsymbol{1}) + \sum_{k = 1}^{n} y^{[k]} \log \nu - \sum_{k = 1}^{n} \log (y^{[k]}!) \nonumber \\
& \;\;\;+ \sum_{i = 1}^{m}S_i \log \alpha_i  + \sum_{i = 1}^{m}\sum_{j = 1}^{m} N_{ij} \log p_{ij}. \label{eqn:logf}
\end{align} 

\paragraph{Maximum-likelihood estimators}
To obtain closed-form expressions for the maximum-likelihood estimators $\hat{\bs{\theta}}$  is not straightforward. Applying the Karush-Kuhn-Tucker approach (see Chapter 12 in Nocedal and Wright~(2000)\nocite{nocedalwright}), it can be verified that the maximization problem 
\begin{align*} 
&\max_{\bs{\theta}} \; \log f(\bs{x} \; | \; \bs{\theta}) \\
\intertext{subject to}
\boldsymbol{\alpha}\boldsymbol{1} = 1, \qquad P \bs{1} \leq \bs{1}, \qquad &\nu > 0, \qquad p_{ij}, \alpha_i \geq 0 \quad \mbox{for } i,j = 1, \ldots, m,
\end{align*}
has the associated Lagrangian 
\begin{align*}
\mathcal{L}(\bs{\theta}, \lambda, \bs{\mu}) =   \log f(\bs{x} \; | \; \bs{\theta}) - \lambda h(\bs{\theta}) - \sum_{i = 1}^{2m +1}\mu_i g_i(\bs{\theta}), 
\end{align*} 
where 
$$ 
\begin{array}{rll} 
h(\bs{\theta}) \hspace*{-0.2cm} & = \bs{\alpha} \bs{1} - 1, \\
g_i(\bs{\theta})  \hspace*{-0.2cm}  & = 1 - \sum\limits_{j = 1}^m p_{ij} & \mbox{ for } i = 1, \ldots, m, \\
                         & = \alpha_{i - m} & \mbox{ for } i = m + 1, \ldots, 2m, \\
                         & = \nu & \mbox{ for } i = 2m + 1, 
\end{array}
$$  
and $\lambda$ and $\bs{\mu} = (\mu_1, \ldots, \mu_{2m + 1}) \geq \bs{0}$ denote the Lagrangian multipliers associated with the equality constraint $h(\bs{\theta}) = 0$, the inequality constraints $g_i(\bs{\theta}) \geq 0$ for $i = 1, \ldots, 2m$ and
$g_{2m + 1}(\bs{\theta}) > 0$, respectively. The KKT conditions, which are first-order necessary conditions for constrained optimization problems, imply that the maximum-likelihood estimators $\hat{\bs{\theta}} = (\hat{\nu}, \hat{\bs{\alpha}}, \hat{P})$ must satisfy the following constraints 
\begin{align}
& \hat{\boldsymbol{\alpha}} e^{\hat{\nu}\hat{P}}(\hat{\nu} \hat{P} \boldsymbol{1} - \frac{\sum_{k = 1}^{n} y^{[k]}}{n} \boldsymbol{1}) 
 = 0\label{eqn:confornu}  \\          
& \hat{\alpha}_i = \frac{S_i}{\hat{\eta}_i}(\sum_{j = 1}^m\frac{S_j}{\hat{\eta}_j})^{-1} \quad \mbox{ for } i = 1, \ldots, m, \label{eqn:conforalpha}  \\
& \frac{n \hat{\bs{\alpha}}}{\hat{\bs{\alpha}} e^{\hat{\nu}\hat{P}} \bs{1}}  \int_0^{\hat{\nu}} e^{(\hat{\nu} - u)\hat{P}}  \bs{e}_{i}\bs{e}\tp_j  e^{u\hat{P}}\ud u \bs{1}  - \frac{N_{ij}}{\hat{p}_{ij}}  \leq 0,  \label{eqn:conforpij} 
\end{align} 
for $i,j = 1, \ldots, m$, where $\hat{\eta}_i = \boldsymbol{e}_i\tp e^{\hat{\nu} \hat{P}}\boldsymbol{1}$ and $\boldsymbol{e}_i$ is the column vector of size~$m$ with the $i$th component being 1 and all other components being 0. 

Recall from Remark~\ref{rem:stoch}, that if $P$ is stochastic then the PH-Poisson distribution with representation $(\nu, \bs{\alpha}, P)$ is a Poisson distribution with parameter~$\nu$. In this case, the constraints~\eqref{eqn:confornu}--\eqref{eqn:conforpij} simplify considerably: the first implies that $\hat{\nu}= \sum_{k = 1}^{n} y^{[k]}/n$, the well-known maximum-likelihood estimator for the parameter of a Poisson distribution; the second becomes $\hat{\alpha}_i = S_i/n$, the maximum-likelihood estimator for the initial vector of a discrete PH distribution~(see Asmussen et al. (1996))\nocite{Asmussen96}; and the third reduces to 
\begin{align}
n\hat{\bs{\alpha}}  \int_{0}^{\hat{\nu}} e^{x(\hat{P} - I)} \ud x \bs{e}_i - \frac{N_{ij}}{\hat{p}_{ij}} \leq 0, \nonumber 
\end{align}  
{or, equivalently, 
\begin{align} 
 \hat{\nu} \hat{p}_{ij} \hat{\bs{\alpha}}  \int_{0}^{1} e^{\hat{\nu}(\hat{P} - I)x}\ud x \bs{e}_i - \frac{N_{ij}}{n} \leq 0, \label{eqn:conforpij2} 
\end{align} 
for $i, j = 1, \ldots, m$. As $\hat{P}$ is stochastic, summing the left-hand side of \eqref{eqn:conforpij2} over $i$ and $j$ gives us 
\begin{align*} 
\hat{\nu}  \hat{\bs{\alpha}}  \int_{0}^{1} e^{\hat{\nu}(\hat{P} - I)x}\ud x \bs{1} - 1/n \sum_{i,j = 1}^m N_{ij} = \hat{\nu} - 1/n \sum_{i,j = 1}^m N_{ij} = 0,
\end{align*} 
which implies that \eqref{eqn:conforpij2} is an equality for all $i, j = 1, \ldots, m$.}

\paragraph{Conditional expectation}

Thanks to the linear nature of $\log f(\bsautre X\,|\,\vc\theta)$ in the unobserved data $\bsautre Z = \{S_i, N_{ij}: i, j = 1, \ldots, m\}$, the computation of the conditional expectation of $\log f(\bsautre X\,|\,\vc\theta^{(s)})$ at the $(s + 1)$th iteration reduces to the computation of $\E [\bsautre Z\,|\,\vc y,{\vc\theta}^{(s)}]$: 
%
\begin{align} 
\E [S_i\,|\, \vect y, \vect\theta^{(s)}]& = \sum_{k = 1}^{n} \E [\mathds{1}_{\{\varphi_0^{[k]}=i\}}\,|\,y^{[k]},\vect\theta^{(s)}] \nonumber \\
& = \sum_{k = 1}^{n} \dfrac{\P [\varphi_0^{[k]}=i\,|\,\vc\theta^{(s)}]\,\P [Y^{[k]}=y^{[k]}\,|\,\vc\theta^{(s)},\varphi_0^{[k]}=i]}{\P [Y^{[k]}=y^{[k]}\,|\,\vc\theta^{(s)}]}\nonumber \\
& = \sum_{k = 1}^{n} \frac{\alpha_i^{(s)}\,\boldsymbol{e}_i^{\mbox{\tiny T}}\, (P^{(s)})^{y^{[k]}}\boldsymbol{1}}{\boldsymbol{\alpha}^{(s)}\,(P^{(s)})^{y^{[k]}}\,\boldsymbol{1}} \quad \mbox{ for } i = 1, \ldots, m,
\end{align} 
and 
\begin{align} 
& \E [N_{ij}\,|\, \vect y, \vect\theta^{(s)}] = \sum_{k=1}^n\sum_{t=1}^{y^{[k]}} \E [\mathds{1}_{\{\varphi_{t-1}^{[k]}=i,\varphi_{t}^{[k]}=j \}}\,|\,y^{[k]},\vect\theta^{(s)}] \nonumber\\
\displaybreak
& = \sum_{k = 1}^{n}\sum_{t=1}^{y^{[k]}} \dfrac{\P [\varphi_{t-1}^{[k]}=i\,|\,\vc\theta^{(s)}]\,\P [\varphi_{t}^{[k]}=j\,|\,\vc\theta^{(s)},\varphi_{t-1}^{[k]}=i]\,\P [Y^{[k]}=y^{[k]}\,|\,\vc\theta^{(s)},\varphi_t^{[k]}=j]}{\P [Y^{[k]}=y^{[k]}\,|\,\vc\theta^{(s)}]}\nonumber \\
&=  \sum_{k = 1}^{n} \frac{\boldsymbol{\alpha}^{(s)}\,\sum_{t = 1}^{y^{[k] }}\,(P^{(s)})^{t-1}\,\boldsymbol{e}_i\, p_{ij}^{(s)}\, \boldsymbol{e}^{\mbox{\tiny T}}_j\,(P^{(s)})^{y^{[k]} - t}\,\boldsymbol{1}}{\boldsymbol{\alpha}^{(s)}\,(P^{(s)})^{y^{[k]}}\,\boldsymbol{1}} \quad \mbox{ for } i,j = 1, \ldots, m.
\end{align} 

\paragraph{New estimates}

In the M-step, we obtain the new estimates ${\vc\theta}^{(s+1)} = (\nu^{(s + 1)},\vc\alpha^{(s + 1)},P^{(s + 1)})$ by maximizing the log-likelihood \eqref{eqn:logf} where $\{S_i,N_{ij}: i,j=1,\ldots,p\}$ are replaced by their conditional expectations $\E [S_i\,|\,\vc y, {\vc\theta}^{(s)}]$ and $\E [N_{ij}\,|\,\vc y, {\vc\theta}^{(s)}]$ evaluated in the E-step. The maximization problem to be solved in this step is as follows 
\begin{align*} 
\max_{\bs{\theta}} \ & \log f(\vc y, \E [S_i\,| \bs{\theta}^{(s)}], \E [N_{ij}\,| \bs{\theta}^{(s)}]\,|\bs{\theta}) \label{eqn:objf} \\
\intertext{subject to }
\boldsymbol{\alpha}\boldsymbol{1} = 1, \qquad P \bs{1} \leq \bs{1}, &\qquad \nu > 0,  \qquad p_{ij}, \alpha_i \geq 0 \quad \mbox{for } i,j = 1, \ldots, m.
\end{align*}

\vspace{1\baselineskip}
We implemented the EM algorithm in MATLAB and experimented with samples simulated from different PH-Poisson distributions. Below are the results of one such experiment. 

\begin{exemple} \rm

We used the PH-Poisson distribution $(\nu;\bs{\alpha}, P)$ given in Example~\ref{exemple:transition} to generate a sample with 1500 observations. The chosen initial parameters are $\nu^{(0)} = 10$, $ \bs{\alpha}^{(0)} = \vligne{0.1 & 0.2 & 0.4 & 0.2 & 0.1}$ and 
$$  P^{(0)} = \diag(0.5, 0.3, 0.5, 0.7, 0.1).$$
The estimated parameters obtained after 25 iterations of the EM algorithm are $\nu^{(25)} =  20.4290$, $ \bs{\alpha}^{(25)} =  \vligne{0.9054 & 0.060 &Ê0.0335 & 0.0000 & 0.0010} $, and 
$$ 
P^{(25)} = \left[\begin{array}{ccccc}
    Ê Ê0.2401 & Ê Ê0.0000 Ê&  Ê0.0000 Ê&  Ê0.0000  &Ê Ê0.0000 \\
 Ê Ê0.0000 Ê& Ê0.2610 Ê & Ê0.0000 Ê&  Ê0.0000 Ê&  Ê0.0000 \\
 Ê Ê0.0000 Ê& Ê0.0000 Ê&  Ê0.4543 Ê&  Ê0.0000 Ê&  Ê0.0000 \\
 Ê Ê0.0000 Ê& Ê0.0000 Ê & Ê0.0000 Ê&  Ê0.9088 Ê & Ê0.0000 \\
 Ê Ê0.9939 Ê &Ê0.0026 Ê&  Ê0.0000 Ê &  Ê0.0000  &Ê Ê0.0035 \\
 \end{array}\right]. $$
The Manhattan norm $||\cdot ||_1$ of the difference between the true density and the empirical data is 0.1109, between the true density and the estimated density is 0.1043, and between the empirical data and the estimated density is 0.1400. We plot four densities in Figure~\ref{fig:densities}: that for the true PH-Poisson distribution, the empirical data, the initial density and the estimated density. 

\begin{figure}[!tbp] 
\centering
\includegraphics[scale=0.41]{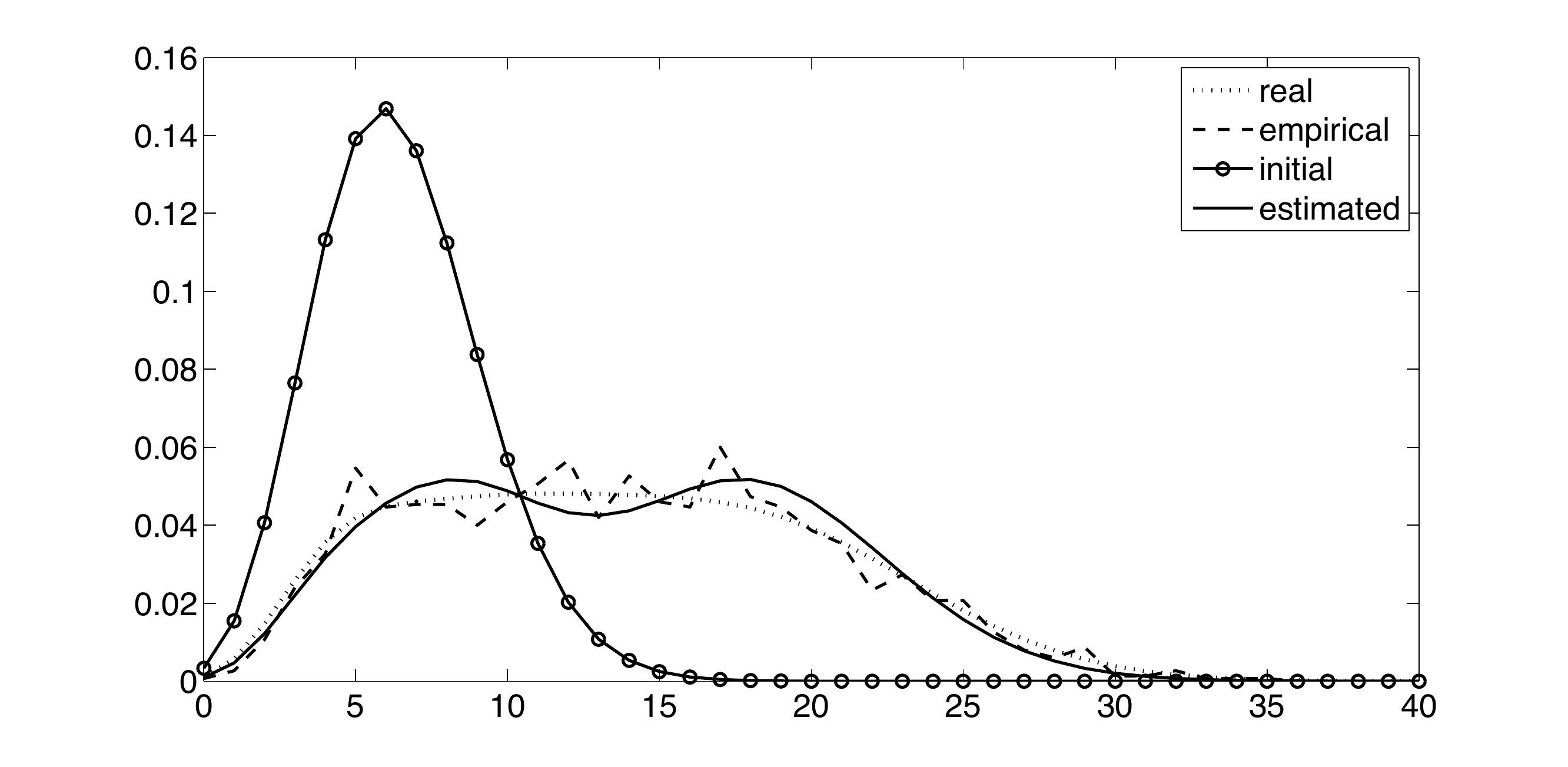} 
\caption{Density function for the true PH-Poisson distribution (the dotted curve), empirical data (the dashed curve), the initial density (the curve marked with $\circ$) and the estimated density (the continuous curve).} \label{fig:densities}
\end{figure}
\end{exemple} 
It is well-known that although the sequence $\{\bs{\theta}^{(s)}\}_{s \geq 1}$ computed with the EM algorithm always converges, it  does not always converge to the maximum-likelihood estimator $\hat{\bs{\theta}}$, but possibly to some local maximum or stationary value of $\log f(\bsautre{X}|\bs{\theta})$. The warranty of global convergence for the EM algorithm depends on properties of
the conditional density of the incomplete data $\bsautre{Y}$ given $\bs{\theta}$, and sometimes also on the starting point $\bs{\theta}^{(0)}$. We refer to Dempster et al.~(1977)\nocite{dlr77} for further details on the EM algorithm, and to Wu~(1983)\nocite{wu1983} for its convergence properties. 

Our experiments were performed using the MATLAB optimization routine \textsf{fmincon} to solve the maximization problem in the M-step.  They indicated that the results were highly sensitive to the choice of $\bs{\theta}^{(0)}$. When the starting point was chosen randomly, we observed that the EM algorithm often converged to a Poisson distribution with parameter $\sum_{k = 1}^{n} y^{[k]}/n$, even if this was a rather poor fit for the given sample. Convergence to a good fit was obtained with a starting point that either shares the same structure of zeros with the true parameters $\bs{\alpha}$ and $P$, or has a strictly positive $\bs{\alpha}^{(0)}$ and a diagonal matrix $P^{(0)}$---a mixture of Poisson distributions. 

The latter choice is obviously more practical when the structure of the true parameters is not known a priori. Empirically, a diagonal $P^{(0)}$ proved to be a good starting point even if the true matrix $P$ is not diagonal. Note that, unlike its counterpart for fitting discrete Phase-type distributions in Asmussen et al. (1996)\nocite{Asmussen96}, the EM algorithm for fitting PH-Poisson distributions does not necessarily preserve the initial structure. This is due to the term $-n \log (\bs{\alpha}e^{\nu P} \bs{1})$ in (\ref{eqn:logf}). Consequently, when starting with a diagonal $P^{(0)}$ the EM algorithm does not necessarily converge to a diagonal $P$. An interesting question for future research is to explain why mixtures of Poisson distributions serve as good starting points in the EM algorithm for fitting Phase-type Poisson distributions.

\subsection*{Acknowledgment}

All three authors thank the Minist\`ere de la Communaut\'e fran\c{c}aise de
Belgique for funding this research through the ARC grant AUWB-08/13--ULB~5.
The first and third authors also acknowledge the Australian Research Council
for funding part of the work through the Discovery Grant DP110101663.

\bigskip

\bibliographystyle{abbrv}
\bibliography{matrixPoisson}

\appendix
\end{document}